\documentclass[10pt]{amsart}
\usepackage{amsmath,graphicx,amssymb,amsthm}
\usepackage{color}
\def\A{\mathcal{A}}
\def\B{\mathcal B}

\def\B{\mathcal B}

\def\amslatex{$\mathcal{A}\kern-.1667em\lower.5ex\hbox{$\mathcal{M}$}\kern-.125em\mathcal{S}$-\LaTeX}

\newtheorem{set}{set}[section]
\newtheorem{Corollary}[set]{Corollary}

\newtheorem{Example}[set]{Example}

\newtheorem{Lemma}[set]{Lemma}

\newtheorem{Remark}[set]{Remark}
\newtheorem{Theorem}[set]{Theorem}
\newcommand{\define}{\mathrel{\hbox{$\equiv$\hskip -.90em \lower .47ex \hbox{$\leftharpoondown$}}}}
\newcommand{\enifed}{\mathrel{\hbox{$\equiv$\hskip -.90em \lower .47ex \hbox{$\rightharpoondown$}}}}

\begin{document}
\title {\bf   On  Bi-free Multiplicative Convolution}
\author{Mingchu Gao}
\address{School of Mathematics and Information Science,
Baoji University of Arts and Sciences,
Baoji, Shaanxi 721013,
China; and
Department of Mathematics,
Louisiana College,
Pineville, LA 71359, USA}
\email{mingchu.gao@lacollege.edu}

\date{}
\begin{abstract}
In this paper, we study the partial bi-free  $S$-transform of a pair $(a,b)$ of random variables, and the $S$-transform of the   $2\times 2$ matrix-valued random variable $\left(\begin{matrix}a&0\\0&b\end{matrix}\right)$ associated with $(a,b)$ when restricted to upper triangular $2\times 2$ matrices. We first derive an explicit expression of  bi-free multiplicative convolution (of probability measures on the 2-dimensional torus $\mathbb{T}^2=\{(s,t)\in\mathbb{C}^2:|s|=1=|t|\}$, or on $\mathbb{R}^2_+$ in $\mathbb{C}^2$) from a subordination equation for bi-free multiplicative convolution.  We then show that, when $(a_1, b_1)$ and $(a_2,b_2)$ are bi-free,  the $S$-transforms of $X_1=\left(\begin{matrix}a_1&0\\0&b_1\end{matrix}\right)$, $X_2=\left(\begin{matrix}a_2&0\\0&b_2\end{matrix}\right)$ satisfy Dykema's twisted multiplicative equation for free operator-valued random variables if and only if at least one of the two partial bi-free  $S$-transforms of the pairs of random variables is the constant function 1 in a neighborhood of $(0,0)$.  This is the case if and only if one of the   two pairs, say $(a_1,b_1)$,  has factoring two-band moments (that is, $\varphi(a_1^mb_1^n)=\varphi(a_1^m)\varphi(b_1^n)$, for all $m,n=1, 2, \cdots$). We thus find a lot of bi-free pairs of random variables to which the $S$-transforms of the corresponding matrix-value random variables do not satisfy Dykema's twisted multiplicative formula. Finally, if both $(a_1,b_1)$ and $(a_2,b_2)$ have factoring two-band moments, we prove that the $\Psi$-transforms of $X_1$, $X_2$, and $X_1X_2$ satisfy a subordination equation.
\end{abstract}
\maketitle
{\bf AMS Mathematics Subject Classification (2010)} 46L54.

{\bf Key words and phrases} Bi-free Multiplicative Convolution, $S$-transforms, Analytic subordination.
\section*{Introduction}
Voiculescu introduced the concept of {\sl freeness for pairs of faces of random variables} in \cite{DV1}, initiating a new research field in free probability, {\sl bi-free probability}.   In his second paper on bi-free probability \cite{DV2}, Voiculescu provided a partial bi-free $R$-transform of a pair $(a,b)$ of random variables  to linearize bi-free additive convolution of compactly supported  probability measures on $\mathbb{R}^2$. The partial bi-free $R$-transform $R_{a,b}(z,w)$ is a formal power series in two complex variables  $z,w$ with bi-free cumulants being its coefficients (2.1 in \cite{DV2}).
Combining the  functional equation $$R_{a,b}(z,w)=1+zR_a(z)+wR_b(w)-\frac{zw}{G_{a,b}(G_{a}^{\langle-1\rangle}(z), G_{b}^{\langle-1\rangle}(w))}, $$ for $z,w$ near $0$, where $G_a^{\langle-1\rangle}(z)$ is the inverse function of $G_a(z)$ (Theorem 2.4 in \cite{DV2}),  with the additive property of $R_{a,b}(z,w)$ $$R_{a_1+a_2, b_1+b_2}(z,w)=R_{a_1, b_1}(z,w)+R_{a_2, b_2}(z,w), $$ for $|z|+|w|$ is sufficiently small, whenever $(a_1, b_1)$ and $(a_2, b_2)$ are bi-free,   the authors of  \cite{BBGS} got  a subordination equation for bi-free additive convolution$$\frac{1}{G_{a_1,b_1}(\omega_{a_1}(z),\omega_{b_1}(w))}+\frac{1}{G_{a_2,b_2}(\omega_{a_2}(z),\omega_{b_2}(w))}=\frac{1}{G_{a_1+a_2}(z)G_{b_1+b_2}(w)}+\frac{1}{G_{a_1+a_2, b_1+b_2}(z,w)}, \eqno (0.1)$$ for $z\in \mathbb{C}\setminus \sigma(a_1+a_2), w\in \mathbb{C}\setminus \sigma(b_1+b_2)$, as an equation of meromorphic functions ($(4)$ in \cite{BBGS}), where $\omega_{a_j}, \omega_{b_j}: \mathbb{C}^+\rightarrow \mathbb{C}^+$ are analytic functions such that $G_{a_1+a_2}(z)=G_{a_j}(\omega_{a_j}(z)),G_{b_1+b_2}(z)=G_{b_j}(\omega_{b_j}(z))$, for $z\in \mathbb{C}^+=\{z\in \mathbb{C}: \Im z>0\}$, and $j=1,2$. The functions $\omega_{a_j}, \omega_{b_j}$ are called subordination functions for free additive convolution.

Let $\Gamma=\left(\begin{matrix}z&\zeta\\0&w\end{matrix}\right)$ be an invertible upper triangular matrix in $M_2(\mathbb{C})$. For $a$ and $b$ in $(\A,\varphi)$, let $X$ be the matrix-valued random variable $\left(\begin{matrix}
                                                                                         a & 0 \\
                                                                                         0& b
                                                                                       \end{matrix}\right)$
in the matrix-valued non-commutative probability space $(M_2(\A), E, M_2(\mathbb{C}))$, where $E:M_2(\A)\rightarrow M_2(\mathbb{C})$ is the conditional expectation defined by $E((a_{ij})_{2\times 2})=(\varphi(a_{ij}))_{2\times 2}$, for $(a_{ij})_{2\times 2}\in M_2(\A)$. Lemma 3.1 in \cite{BBGS} states that if $(a_1, b_1)$ and $(a_2, b_2)$ are bi-free in $(\A,\varphi)$, then $$R_{X_1+X_2}(\Gamma)=R_{X_1}(\Gamma)+R_{X_2}(\Gamma),\eqno (0.2)$$ for all $z,w\in \mathbb{C}$ of sufficiently small absolute value, and all $\zeta\in \mathbb{C}$. The authors of \cite{BBGS} emphasized that this additive formula implies that the matrix-valued random variables $X_1=\left(\begin{matrix}a_1 & 0 \\0 & b_1\end{matrix}\right)$ and $X_2=\left(\begin{matrix}a_2 & 0 \\0 & b_2\end{matrix}\right)$
 ``mimic" freeness in terms of the relations between their analytic transforms when restricted to upper triangular matrices, which, furthermore, implies a subordination result for the $G$-transforms of $X_1$, $X_2$, and $X_1+X_2$ (Proposition 3.4 in \cite{BBGS}).

 In this paper, we study similar questions for bi-free multiplicative convolution. We derive an explicit expression for bi-free multiplicative convolution from a subordination equation for bi-free multiplicative convolution. The authors in \cite{BBGS} proved the additive equation for $R$-transforms of $X_1$ and $X_2$ ((0.2)), and provided a counterexample to show  that $X_1$ and $X_2$ are not free over $M_2(\mathbb{C})$ (Example 3.3 in \cite{BBGS}). An interesting question is whether the $S$-transforms of $X_1$ and $X_2$ satisfy Dykema's twisted multiplicative equation for free random variables with amalgamation (Theorem 1.1 in \cite{KD}) when restricted to upper triangular matrices. We answer this question completely by  providing  sufficient and necessary conditions for  the $S$-transforms of $X_1$ and $X_2$ to satisfy Dykema's equation.   Finally, we get a subordination formula for the $\Psi$-transforms of the matrix-valued random variables associated with bi-free  pairs of random variables in the case of factoring two-band moments.

The paper is organized as follows. Section 1 is devoted to the study of subordination properties of bi-free multiplicative convolution.  Using Voiculescu's multiplicative formula for partial bi-free $S$-transforms (Theorem 2.1 in \cite{DV3}), we derive a subordination equation for bi-free multiplicative convolution (Theorem 1.1), which could be regarded as a multiplicative analogue of above $(0.1)$.
 If  $a_ib_i=b_ia_i$, for $i=1,2$, then the joint distribution $\mu_i$ of $a_i, b_i$ is determined by its two-band moments $\{\varphi(a_i^mb_i^n):m,n=1, 2, \}$ for $i=1,2$. If $(a_1, b_1)$ and $(a_2, b_2)$ are bi-free and $a_ib_i=b_ia_i$, for $i=1,2$, then the distribution of the product pair $(a_1a_2, b_1b_2)$ is determined by its two-band moments.  In this case, we denote the distribution $\mu$ of $(a_1a_2, b_1b_2)$ by $\mu_1\boxtimes\boxtimes\mu_2$. If furthermore, $a_1, a_2, b_1, b_2$ are unitaries in a $C^*$-probability space $(\A, \varphi)$, Huang and Wang \cite{HW} pointed out that the $S$-transform $S_\mu(z,w)$ of a measure $\mu$ cannot determine the distribution $\mu$ uniquely, but its $\Psi$-transform can (discussions in Pages 8 and 11 in \cite{HW}).   We get a formula $(1.3)$  from Theorem 1.1 for calculating the $\Psi$-transform of the  bi-free multiplicative convolution  measure $\mu_1\boxtimes\boxtimes\mu_2$. We prove that $(1.3)$ can recover the distribution measure $\mu_1\boxtimes\boxtimes\mu_2$  (Remark 1.3).  In section 2, we study the $S$-transforms of $X_1$ and $X_2$ when restricted to upper triangular matrices in $M_2(\mathbb{C})$.  We prove that, when $(a_1, b_1)$ and $(a_2, b_2)$ are bi-free,
the $S$-transforms of $X_1$ and $X_2$  satisfy Dykema's twisted multiplicative equation for operator-valued free random variables if and only if at least one of the two partial bi-free $S$-transforms of the pairs of random variables is the constant function 1 in a neighborhood of $(0,0)$ (Theorem 2.3). By Remark 4.4 in \cite{PS1}, Proposition 4.2 in \cite{DV3}, or Remark 2.7 in \cite{PS2}, $S_{a,b}(z,w)=1$ when $z$ and $w$ are near $0$ if and only if $(a,b)$ has factoring two-band moments. Therefore, we get a lot of examples of bi-free pairs $(a_1, b_1)$ and $(a_2, b_2)$ such that $S_{X_1}$ and $S_{X_2}$ do not satisfy Dykema's equation (Example 2.6).  Finally, it is proved that $\Psi$-transforms of ${X_1X_2}, X_1, X_2$ satisfy a subordination equation, if both two pairs have factoring two-band moments (Theorem 2.7).

{\bf Acknowledgements} It is the author's pleasure to thank the referee for carefully reviewing the paper and providing many valuable corrections and suggestions, especially, simplifying the proof of Lemma 2.1.  The author would like thank Dr. Paul Skoufranis at York University in Toronto, Canada, for reminding the author his  work in \cite{PS1} and \cite{PS2}, and giving some helpful comments on the initial version of the paper.

\section{An explicit expression for bi-free multiplicative convolution}
 We shall derive a subordination equation for bi-free multiplicative convolution, from which an explicit formula for calculating bi-free multiplicative convolution is obtained.

  Let $(a,b)$ be a pair of random variables in a non-commutative probability space $(\A,\varphi)$. Voiculescu \cite{DV3} defined the following formal power series
$$G_a(z)=\sum_{n\ge 0}z^{-n-1}\varphi(a^n),\
G_{a,b}(z,w)=\sum_{m,n\ge 0}z^{-m-1}w^{-n-1}\varphi(a^mb^n),
$$
$$
h_a(z)=\sum_{n=0}^{\infty}\varphi(a^n)z^n, \Psi_a(z)=h_a(z)-1,\
H_{a,b}(z,w)=\sum_{m=0,n=0}^{\infty}\varphi(a^mb^n)z^mw^n.
$$
If $\varphi(a)\ne 0$, then $\Psi_a(z)$ has an inverse function $\Psi_a^{\langle-1\rangle}(z)$ near zero. The $S$-transform of $a$ is defined as $$S_a(z)=\frac{z+1}{z}\Psi_a^{\langle-1\rangle}(z).$$ The key property of $S_a(z)$ is that if $a_1$ and $a_2$ are free in $(\A, \varphi)$, and $\varphi(a_1)\varphi(a_2)\ne 0$, then $$S_{a_1a_2}(z)=S_{a_1}(z)S_{a_2}(z)$$ (\cite{DV4}). If $\varphi(a)\ne 0\ne \varphi(b)$, then the partial bi-free $S$-transform $S_{a,b}(z,w)$ of $(a,b)$ is defined as $$S_{a,b}(z,w)=\frac{z+1}{z}\frac{w+1}{w}\left(1-\frac{1+z+w}{H_{a,b}(\Psi_a^{\langle-1\rangle}(z),\Psi_b^{\langle-1\rangle}(w))}\right),$$ for $z,w\ne 0$, and $z,w$ near $0$ (Definition 2.1  in \cite{DV3}).
Huang and Wang \cite{HW} defined the following transforms (the original transforms were defined as integral transforms of measures on the $2$-dimensional torus $\mathbb{T}^2$)
$$\Psi_{a,b}(z,w)=\sum_{m=1, n=1}^{\infty}\varphi(a^mb^n)z^mw^n,  \eta_a(z)=\frac{\Psi_a(z)}{1+\Psi_a(z)}.$$
It is easy to see that $$H_{a,b}(z,w)=\Psi_{a,b}(z,w)+\Psi_a(z)+\Psi_b(w)+1$$ ($(2.5)$ in \cite{HW}). Assume that $\varphi(a)\ne 0\ne \varphi(b)$. Huang and Wang \cite{HW} defined the  $\Sigma$-transform of $(a,b)$
$$ \Sigma_{a,b}(z,w)=S_{a,b}\left(\frac{z}{1-z},\frac{w}{1-w}\right)=\frac{\Psi_{a,b}(\eta_a^{\langle-1\rangle}(z),\eta_b^{\langle-1\rangle}(w))}{zwH_{a,b}(\eta_a^{\langle-1\rangle}(z),\eta_b^{\langle-1\rangle}(w)))},$$ for $z,w\in \Omega_r:=D_r\cup \Delta_r$, where $\eta_a^{\langle-1\rangle}(z)$ is the inverse function of $\eta_a(z)$,  $D_r=\{z\in \mathbb{C}:|z|<r\}, \Delta_r=\{z\in \mathbb{C}: |z|>\frac{1}{r}\}$, for some $0<r<1$.

Our first result is a subordination equation for bi-free multiplicative convolution, which is a multiplicative analogue of  the subordination result $(0.1)$  for bi-free additive  convolution.
\begin{Theorem}Let $(a_1, b_1)$ and $(a_2, b_2)$ be bi-free pairs of unitaries (or non-zero non-negative elements) in a $C^*$-probability space $(\A, \varphi)$. If $\varphi(a_i)\ne 0\ne \varphi(b_i)$, for $i=1,2$,  then there are analytic functions $\omega_{a_i}, \omega_{b_i}:\mathbb{D}\rightarrow \mathbb{D}$ (or $\omega_{a_i}, \omega_{b_i}:\mathbb{C}\setminus \mathbb{R}^+\rightarrow \mathbb{C}\setminus \mathbb{R}^+$) such that
\begin{align*}&\frac{1}{\Psi_{a_1a_2,b_1b_2}(z,w)}+\frac{1}{1+\Psi_{a_1a_2}(z)+\Psi_{b_1b_2}(w)}\\
=&\frac{\frac{1}{1+\Psi_{a_1}(\omega_{a_1}(z))+\Psi_{b_1}(\omega_{b_1}(w))}+\frac{1}{\Psi_{a_1}(\omega_{a_1}(z))\Psi_{b_1}(\omega_{b_1}(w))}}{\left( \frac{(1+\Psi_{a_1}(\omega_1(z)))(1+\Psi_{b_1}(\omega_1(w)))}{\Psi_{a_1}(\omega_1(z))\Psi_{b_1}(\omega_{b1}(w))}\right)^2\frac{\Psi_{a_1, b_1}(\omega_{a_1}(z), \omega_{b_1}(w))}{H_{a_1,b_1}(\omega_{a_1}(z),\omega_{b_1}(w))}\frac{\Psi_{a_2, b_2}(\omega_{a_2}(z), \omega_{b_2}(w))}{H_{a_2,b_2}(\omega_{a_2}(z),\omega_{b_2}(w))}},
\end{align*}
for $z,w\in \mathbb{C}\setminus \mathbb{R}^+$ near zero, where $\mathbb{R}^+=\{x\in \mathbb{R}: x\ge 0\}$,  $\mathbb{D}$ is the open unit disk of $\mathbb{C}$.
\end{Theorem}
\begin{proof}By the above definition of partial bi-free  $S$-transforms,  if $\varphi(a)\ne 0\ne \varphi(b)$, and $z,w$ are near $0$ in $\mathbb{C}$ and $z\ne 0\ne w$, we have
\begin{align*}
S_{a,b}(\Psi_a(z), \Psi_b(w))&=\left(1+\frac{1+\Psi_a(z)+\Psi_b(w)}{\Psi_a(z)\Psi_b(w)}\right)\frac{\Psi_{a,b}(z,w)}{H_{a,b}(z,w)}\\
&=\frac{1+\frac{1+\Psi_a(z)+\Psi_b(w)}{\Psi_a(z)\Psi_b(w)}}{1+\frac{1+\Psi_a(z)+\Psi_b(w)}{\Psi_{a,b}(z,w)}}=\frac{\frac{1}{1+\Psi_a(z)+\Psi_b(w)}+\frac{1}{\Psi_a(z)\Psi_b(w)}}
{\frac{1}{1+\Psi_a(z)+\Psi_b(w)}+\frac{1}{\Psi_{a,b}(z, w)}}.
\end{align*}
It follows that
$$\frac{1}{\Psi_{a,b}(z, w)}+\frac{1}{1+\Psi_a(z)+\Psi_b(w)}=\frac{\frac{1}{1+\Psi_a(z)+\Psi_b(w)}+\frac{1}{\Psi_a(z)\Psi_b(w)}}
{S_{a,b}(\Psi_a(z),\Psi_b(w))}.$$
If $(a_1, b_1)$ and $(a_2, b_2)$ are bi-free, and $\varphi(a_i)\ne 0\ne \varphi(b_i)$, for $i=1, 2$, Voiculescu proved that $$S_{a_1a_2, b_1b_2}(z,w)=S_{a_1, b_1}(z,w)S_{a_2, b_2}(z,w),\eqno (1.1)$$ for $z,w\in \mathbb{C}\setminus \{0\}$ near zero (Theorem 2.1 in \cite{DV3}).  If furthermore, $a_1, a_2, b_1$, and $b_2$ are unitaries (or $a_1, a_2, b_1$, and $b_2$ are non-zero non-negative elements) in a $C^*$-probability space $(\A,\varphi)$, by the well-known subordination theorems (\cite{PB}), there are analytic functions $\omega_{a_i}, \omega_{b_i}:\mathbb{D}\rightarrow \mathbb{D}$ (or $\omega_{a_i}, \omega_{b_i}:\mathbb{C}\setminus \mathbb{R}^+\rightarrow \mathbb{C}\setminus \mathbb{R}^+$ ) such that $$\Psi_{a_1a_2}(z)=\Psi_{a_i}(\omega_{a_i}(z)),\  \Psi_{b_1b_2}(z)=\Psi_{b_i}(\omega_{b_i}(z)),$$ for $z\in \mathbb{D}$ (or $z\in \mathbb{C}\setminus \mathbb{R}^+$), for $i=1,2$.

We thus get, for bi-free pairs $(a_1, b_1)$ and $(a_2, b_2)$,
\begin{align*}
&\frac{1}{\Psi_{a_1a_2,b_1b_2}(z,w)}+\frac{1}{1+\Psi_{a_1a_2}(z)+\Psi_{b_1b_2}(w)}\\
=&\frac{\frac{1}{1+\Psi_{a_1a_2}(z)+\Psi_{b_1b_2}(w)}+\frac{1}{\Psi_{a_1a_2}(z)\Psi_{b_1b_2}(w)}}{S_{a_1,b_1}(\Psi_{a_1a_2}(z),\Psi_{b_1b_2}(w))S_{a_2,b_2}(\Psi_{a_1a_2}(z),\Psi_{b_1b_2}(w))}\\
=&\frac{\frac{1}{1+\Psi_{a_1}(\omega_{a_1}(z))+\Psi_{b_1}(\omega_{b_1}(w))}+\frac{1}{\Psi_{a_1}(\omega_{a_1}(z))\Psi_{b_1}(\omega_{b_1}(w))}}{S_{a_1,b_1}(\Psi_{a_1}(\omega_{a_1}(z)),\Psi_{b_1}(\omega_{b_1}(w))S_{a_2,b_2}(\Psi_{a_2}(\omega_{a_2}(z)),\Psi_{b_2}(\omega_{b_2}(w))}\\
=&\frac{\frac{1}{1+\Psi_{a_1}(\omega_{a_1}(z))+\Psi_{b_1}(\omega_{b_1}(w))}+\frac{1}{\Psi_{a_1}(\omega_{a_1}(z))\Psi_{b_1}(\omega_{b_1}(w))}}{\left( \frac{(1+\Psi_{a_1}(\omega_1(z)))(1+\Psi_{b_1}(\omega_1(w)))}{\Psi_{a_1}(\omega_1(z))\Psi_{b_1}(\omega_{b1}(w))}\right)^2\frac{\Psi_{a_1, b_1}(\omega_{a_1}(z), \omega_{b_1}(w))}{H_{a_1,b_1}(\omega_{a_1}(z),\omega_{b_1}(w))}\frac{\Psi_{a_2, b_2}(\omega_{a_2}(z), \omega_{b_2}(w))}{H_{a_2,b_2}(\omega_{a_2}(z),\omega_{b_2}(w))}},\\
\end{align*}
for $z,w\in \mathbb{C}\setminus \mathbb{R}^+$ near zero.
\end{proof}
\begin{Remark}If $(a_1, b_1)$ and $(a_2, b_2)$ are bi-free and $a_ib_i=b_ia_i$, for $i=1,2$, then the joint distribution $\mu_i$ of $a_i, b_i$ is determined by its two-band moments $\{\varphi(a_i^mb_i^n):m,n=1, 2, \}$ for $i=1,2$. So is the distribution of the product pair $(a_1a_2, b_1b_2)$. Indeed, $$\varphi((a_1a_2)^{m_1}(b_1b_2)^{n_1}\cdots (a_ia_2)^{m_k}(b_1b_2)^{n_k})=\varphi((a_1a_2)^m(b_1b_2)^n),$$ where $m=\sum_{i=1}^{k}m_i$,  $n=\sum_{i=1}^{k}n_i$, and we used the fact that $a_ib_j=b_ja_i$, for $i\ne j$ if $(a_1, b_1)$ and $(a_2, b_2)$ are bi-free. In this case, we denote the distribution $\mu$ of $(a_1a_2, b_1b_2)$ by $\mu_1\boxtimes\boxtimes\mu_2$.  It follows that we can get the $\Psi$-transform for the new probability measure $\mu_1\boxtimes\boxtimes\mu_2$. Indeed, let $A(z,w)=1+\Psi_{a_1}(\omega_{a_1}(z))+\Psi_{b_1}(\omega_{b_1}(w))$ and $B(z,w)=\Psi_{a_1}(\omega_{a_1}(z))\Psi_{b_1}(\omega_{b_1}(w))$, we can rewrite the equation in Theorem 2.1 as
$$\Psi_{\mu_1\boxtimes\boxtimes\mu_2}(z,w)
=\left(\frac{\frac{1}{A(z,w)}+\frac{1}{B(z,w)}}{\left(1+\frac{A(z,w)}{B(z,w)}\right)^2\frac{\Psi_{a_1, b_1}(\omega_{a_1}(z), \omega_{b_1}(w))}{H_{a_1,b_1}(\omega_{a_1}(z),\omega_{b_1}(w))}\frac{\Psi_{a_2, b_2}(\omega_{a_2}(z), \omega_{b_2}(w))}{H_{a_2,b_2}(\omega_{a_2}(z),\omega_{b_2}(w))}} -\frac{1}{A(z,w)}\right)^{-1}, \eqno(1.2)$$ for $z,w\in \mathbb{C}\setminus \mathbb{R}^+$ near zero.
\end{Remark}
\begin{Remark}
If  $(a_1, b_1)$ and $(a_2, b_2)$ are bi-free, $a_ib_i=b_ia_i$, for $i=1,2$, and $a_1, a_2, b_1, b_2$ are unitaries in a $C^*$-probability space $(\A,\varphi)$ such that $\varphi(a_i)\varphi(b_i)\varphi(a_ib_i)\ne 0$, for $i=1, 2$, then the distributions of $(a_1, b_1)$, $(a_2, b_2)$, and $(a_1a_2, b_1b_2)$ are probability measures on $\mathbb{T}^2$. In this case, Huang and Wang defined $H$ and $\Psi$ transforms as integrals of the distribution in \cite{HW}:
$$\Psi_\mu(z,w)=\int_{\mathbb{T}^2}\frac{zs}{1-zs}\frac{wt}{1-wt}d\mu(s,t),\  H_\mu(z,w)=\int_{\mathbb{T}^2}\frac{1}{(1-zs)(1-wt)}d\mu(s,t), (z,w)\in (\mathbb{C}\setminus \mathbb{T})^2,$$ where $\mu$ is a probability measure on $\mathbb{T}^2$. The two functions are holomorphic in their domains.
Then   by the proof of Theorem 1.1, for $z,w$ near zero and $z\ne 0\ne w$, we have
\begin{align*}&\frac{1}{\Psi_{a_1a_2,b_1b_2}(z,w)}+\frac{1}{1+\Psi_{a_1a_2}(z)+\Psi_{b_1b_2}(w)}\\
=&\frac{\frac{1}{1+\Psi_{a_1a_2}(z)+\Psi_{b_1b_2}(w)}+\frac{1}{\Psi_{a_1a_2}(z)\Psi_{b_1b_2}(w)}}{S_{a_1,b_1}(\Psi_{a_1a_2}(z),\Psi_{b_1b_2}(w))S_{a_2,b_2}(\Psi_{a_1a_2}(z),\Psi_{b_1b_2}(w))}\\
=&\frac{(1+\Psi_{a_1a_2}(z))(1+\Psi_{b_1b_2}(w))}{(1+\Psi_{a_1a_2}(z)+\Psi_{b_1b_2}(w))(\Psi_{a_1a_2}(z)\Psi_{b_1b_2}(w))S_{a_1,b_1}(\Psi_{a_1a_2}(z),\Psi_{b_1b_2}(w))S_{a_2,b_2}(\Psi_{a_1a_2}(z),\Psi_{b_1b_2}(w))}\\
=&\frac{H_{a_1, b_1}(\omega_{a_1}(z), \omega_{b_1}(w))H_{a_2, b_2}(\omega_{a_2}(z), \omega_{b_2}(w))}{(1+\Psi_{a_1a_2}(z)+\Psi_{b_1b_2}(w))(1+\Psi_{a_1a_2}(z))(1+\Psi_{b_1b_2}(w))}
\frac{\Psi_{a_1}(\omega_{a_1}(z))\Psi_{b_1}(\omega_{b_1}(w))}{\Psi_{a_1, b_1}(\omega_{a_1}(z), \omega_{b_1}(w))\Psi_{a_2, b_2}(\omega_{a_2}(z), \omega_{b_2}(w))}.
\end{align*}
Note that in this case, $\Psi_{a,b}(z,w)=zw\eta_{a,b}(z,w)$,  where
$\eta_{a,b}(z,w)=\varphi(ab(1-za)^{-1}(1-wb)^{-1})))$ is holomorphic in  $(\mathbb{C}\setminus \mathbb{T})^2$, and $\lim_{z,w\rightarrow 0}\eta_{a,b}(z,w)=\varphi(ab)\ne 0$. Similarly, $\Psi_a(z)=z\zeta_a(z)$, where $\zeta_a(z)=\varphi(a(1-za)^{-1})$ is holomorphic in $\mathbb{C}\setminus \mathbb{T}$, and $\lim_{z\rightarrow 0}\zeta_a(z)=\varphi(a)\ne 0$. Now we can continue our calculation,
\begin{align*}
&\frac{1}{\Psi_{a_1a_2,b_1b_2}(z,w)}+\frac{1}{1+\Psi_{a_1a_2}(z)+\Psi_{b_1b_2}(w)}\\
=&\frac{H_{a_1, b_1}(\omega_{a_1}(z), \omega_{b_1}(w))H_{a_2, b_2}(\omega_{a_2}(z), \omega_{b_2}(w))}{(1+\Psi_{a_1a_2}(z)+\Psi_{b_1b_2}(w))(1+\Psi_{a_1a_2}(z))(1+\Psi_{b_1b_2}(w))}\frac{\Psi_{a_1}(\omega_{a_1}(z))\Psi_{b_1}(\omega_{b_1}(w))}{\Psi_{a_1, b_1}(\omega_{a_1}(z), \omega_{b_1}(w))\Psi_{a_2, b_2}(\omega_{a_2}(z), \omega_{b_2}(w))}\\
=&\frac{H_{a_1, b_1}(\omega_{a_1}(z), \omega_{b_1}(w))H_{a_2, b_2}(\omega_{a_2}(z), \omega_{b_2}(w))}{(1+\Psi_{a_1a_2}(z)+\Psi_{b_1b_2}(w))(1+\Psi_{a_1a_2}(z))(1+\Psi_{b_1b_2}(w))}
\frac{\eta_{a_1}(\omega_{a_1}(z))\eta_{b_1}(\omega_{b_1}(w))}{\eta_{a_1, b_1}(\omega_{a_1}(z), \omega_{b_1}(w))\Psi_{a_2, b_2}(\omega_{a_2}(z), \omega_{b_2}(w))}.
\end{align*}

Let $\mu_1$ and $\mu_2$ be the distributions of $(a_1, b_1)$ and $(a_2, b_2)$, respectively. We then have
$$\Psi_{\mu_1\boxtimes\boxtimes\mu_2}(z,w)=\Psi_{a_1a_2, b_1b_2}(z, w)=\frac{F(z,w)}{G(z,w)},\eqno (1.3)$$ for $z\ne0\ne w$ near zero in $\mathbb{C}$,  where
\begin{align*}
&F(z,w)\\
=&(1+\Psi_{a_1a_2}(z)+\Psi_{b_1b_2}(w))(1+\Psi_{a_1a_2}(z))(1+\Psi_{b_1b_2}(w))\\
\times &\eta_{a_1, b_1}(\omega_{a_1}(z), \omega_{b_1}(w))\Psi_{a_2, b_2}(\omega_{a_2}(z), \omega_{b_2}(w))\\
&G(z,w)\\
=& H_{a_1, b_1}(\omega_{a_1}(z), \omega_{b_1}(w))H_{a_2, b_2}(\omega_{a_2}(z), \omega_{b_2}(w))\eta_{a_1}(\omega_{a_1}(z))\eta_{b_1}(\omega_{b_1}(w))\\
-&(1+\Psi_{a_1a_2}(z))(1+\Psi_{b_1b_2}(w))\eta_{a_1, b_1}(\omega_{a_1}(z), \omega_{b_1}(w))\Psi_{a_2, b_2}(\omega_{a_2}(z), \omega_{b_2}(w)).
\end{align*}
Since $F$ and $G$ are holomorphic in the unit bi-disk $\mathbb{D}^2=\{(z,w)\in \mathbb{C}^2: |z|, |w|\le 1\}$, and $$\lim_{z,w\rightarrow 0}G(z,w)=\varphi (a_1)\varphi(b_1)\ne 0,$$ we have that $\frac{F(z,w)}{G(z,w)}$ is a homomorphic function $\mathbb{D}^2$.  The equation $(2.3)$ shows that $\Psi_{a_1, a_2, b_1, b_2}(z,w)=\frac{F(z,w)}{G(z,w)}$ in a neighborhood of $(0,0)$ in $\mathbb{D}^2$. By the zero property theorem for multi-variable holomorphic functions [G, Lemma 24], the  homomorphic function $\Psi_{a_1, a_2, b_1, b_2}(z,w)$ is equal to the homomorphic function $\frac{F(z,w)}{G(z,w)}$ on the whole region $\mathbb{D}^2$.
Now we can repeat the discussion in Page 8 in \cite{HW}:

 Let $$g(z,w)=4\Psi_{a_1a_2,b_1b_2}(z,w)+2(\Psi_{a_1a_2}(z)+\Psi_{b_1b_2}(w))+1.$$ Then $$\Re (\frac{g(z,w)-g(z,1/\overline{w})}{2})=\int_{\mathbb{T}^2}\Re(\frac{1+zs}{1-zs})\Re(\frac{1+wt}{1-wt})d\mu_1\boxtimes\boxtimes\mu_2(s,t), (z,w)\in \mathbb{D}^2,$$ recovers the values of the Poisson integral of the measure $d\mu_1\boxtimes\boxtimes\mu_2(1/s, 1/t)$, determining $\mu_1\boxtimes\boxtimes\mu_2$.

It was pointed out in \cite{HW} that the $S$- or $\Sigma$- transform of a measure $\mu$ is insufficient to determine the measure $\mu$ (See the discussion in Page 11 in \cite{HW}). Our discussion above shows that  formula $(1.3)$ provides a complete solution to determining $\mu_1\boxtimes\boxtimes\mu_2$, as it determines the marginals together with the $S$-transform,  while Voiculescu's multiplicative identity $(1.1)$ provides a method for computing the $S$-transform of $\mu_1\boxtimes\boxtimes\mu_2$,  for two probability measures $\mu_1$ and $\mu_2$ on $\mathbb{T}^2$ providing $$m_{1,1}(\mu_i)=\int_{\mathbb{T}^2}std(\mu_i(s,t))\ne 0,$$
$$ m_{1,0}(\mu_i)=\int_{\mathbb{T}^2}sd(\mu_i(s,t))\ne 0,\  m_{0,1}(\mu_i)=\int_{\mathbb{T}^2}td(\mu_i(s,t))\ne 0,$$ for $i=1,2$.
\end{Remark}

\section{The  $S$-transform of the $2\times 2$ matrix $X$ associated with $(a,b)$}

 Let $\Gamma=\left(\begin{matrix}z&\zeta\\0&w\end{matrix}\right)$ be an invertible upper triangular matrix in $M_2(\mathbb{C})$. For $a$ and $b$ in $(\A,\varphi)$, let $X$ be the matrix-valued random variable $\left(\begin{matrix}
                                                                                         a & 0 \\
                                                                                         0& b
                                                                                       \end{matrix}\right)$
in the matrix-valued non-commutative probability space $(M_2(\A), E, M_2(\mathbb{C}))$.
 We will study the $S$-transform of $X$, and find conditions under which $X_1$ and $X_2$ satisfy Dykema's twisted multiplicative equation for free operator-valued random variables,  if $(a_1, b_1)$ and $(a_2, b_2)$ are bi-free, where $X_j=\left(\begin{matrix}a_j&0\\0&b_j\end{matrix}\right)$, $j=1,2$.

We will call an {\sl operator-valued noncommutative probability space} a triple $(\A, E, \B)$, where $\B\subset \A$ is an inclusion of von Neumann algebras, $E:\A\rightarrow \B$ is a unit-preserving conditional expectation. Elements in $\A$ are called {\sl operator-valued random variables}. Two subalgebras $\A_1, \A_2$ of $\A$ containing $\B$ are called {\sl free over $\B$} if $$E(x_1x_2\cdots x_n)=0$$ whenever $n\in \mathbb{N}, x_j\in \A_{i_j}$ satisfy $E(x_j)=0$ and $i_j\ne i_{j+1}$, $1\le j\le n-1$. Two random variables $x,y\in \A$ are free over $\B$ if the algebras $B\langle x\rangle$ and $\B\langle y\rangle$ generated by $\B$ and $x$, and $\B$ and $y$, respectively, are free over $\B$ (Section 2 in \cite{BSTV}).

For an operator-valued non-commutative probability space $(\A, E, \B)$, define $\mathbb{H}^+(\B)=\{b\in \B: \Im b=\frac{b-b^*}{2i}>0\}$, and an analytic mapping $\Psi_x(b)=E((1-bx)^{-1}-1)$, for $x\in \A, b\in \mathbb{H}^+(\B)$. The mapping $\Psi_x(b)$ has an inverse $\Psi_x^{\langle-1\rangle}$ around zero, if $E(x)$ is invertible in $\B$ (Section 2 of \cite{BSTV}). Dykema \cite{KD} defined the $S$-transform $S_X$ for an operator-valued random variable $X\in (\A, E, \B)$ as follows $$S_X(b)=b^{-1}(1+b)\Psi_X^{\langle-1\rangle}(b),$$ when $b$ is invertible and $\|b\|$ is small enough. Dykema proved in Theorem 1.1 of \cite{KD} that, whenever $E(x)$ and $E(y)$ are both invertible in $\B$, $$S_{xy}(b)=S_y(b)S_x(S_y(b)^{-1}bS_y(b)),\eqno (2.1)$$ for invertible $b\in \B$ and $\|b\|$  small enough.

\begin{Lemma} Let $X_j=\left(\begin{matrix}a_j&0\\0&b_j\end{matrix}\right)$, for $j=1,2$. If $(a_1,b_1)$ and $(a_2, b_2)$ are bi-free in $(\A, \varphi)$, and $\varphi(a_j)\varphi(b_j)\ne 0$, for $j=1,2$,   then we have $$\lim_{z\rightarrow 0, w\rightarrow 0}S_{X_1X_2}(\Gamma)=\left(\begin{matrix}\frac{1}{\varphi(a_1)\varphi(a_2)}&\zeta \frac{\varphi(a_1)\varphi(a_2)\varphi(b_1)\varphi(b_2)-\varphi(a_1b_1)\varphi(a_2b_2)}{\varphi(a_1)\varphi(a_2)(\varphi(b_1)\varphi(b_2))^2}\\
0&\frac{1}{\varphi(b_1)\varphi(b_2)}\end{matrix}\right).\eqno(2.2)$$
\end{Lemma}
\begin{proof}
Let's calculate the $S$-transform of $X$. First we need to figure out $\Psi_X^{\langle-1\rangle}(\Gamma)$
\begin{align*}
\Psi_X(\Gamma)&=E((1-\Gamma X)^{-1})-1=E\left(\left(\begin{matrix}
                                                      1-za & -\zeta b \\
                                                      0 & 1-wb
                                                    \end{matrix}\right)^{-1}\right)-1\\
&=E\left(\left(\begin{matrix}
                            (1-za)^{-1} & \frac{\zeta}{w}(1-za)^{-1}w b(1-wb)^{-1} \\
                                                      0 & (1-wb)^{-1}
\end{matrix}\right)\right)-1\\
&=\left(\begin{matrix}\varphi((1-za)^{-1})-1 & \frac{\zeta}{w}\varphi((1-za)^{-1}w b(1-wb)^{-1}) \\
                                                      0 & \varphi((1-wb)^{-1})-1
                                                    \end{matrix}\right)
                                                    =\left(\begin{matrix}
                                                      \Psi_a(z) & \frac{\zeta}{w}(H_{a,b}(z,w)-h_a(z))\\
                                                      0 & \Psi_b(w)
                                                    \end{matrix}\right)\\
&=\left(\begin{matrix}
                                                      \Psi_a(z) & \frac{\zeta}{w}(\Psi_{a,b}(z,w)+\Psi_b(w))\\
                                                      0 & \Psi_b(w)
                                                    \end{matrix}\right).
\end{align*}
It implies that the inverse mapping $\Psi_X^{\langle-1\rangle}(\Gamma)$ of $\Psi_X(\Gamma)$ has the following form
$$\Psi_X^{\langle-1\rangle}(\Gamma)=\left(\begin{matrix}
                                                      \Psi_a^{\langle-1\rangle}(z) &\zeta \frac{\Psi_b^{\langle-1\rangle}(w)}{\Psi_{a,b}(\Psi_a^{\langle-1\rangle}(z),\Psi_b^{\langle-1\rangle}(w))+w}\\
                                                      0 & \Psi_b^{\langle-1\rangle}(w)
                                                    \end{matrix}\right).$$
It follows that, whenever $\varphi(a)\ne 0\ne \varphi(b)$,
\begin{align*}
S_X(\Gamma)=&\Gamma^{-1}(1+\Gamma)\Psi_X^{\langle-1\rangle}(\Gamma)
=\left(\begin{matrix}\frac{z+1}{z}&-\frac{\zeta}{zw}\\0&\frac{w+1}{w}\end{matrix}\right)\left(\begin{matrix}
                                                      \Psi_a^{\langle-1\rangle}(z) &\zeta \frac{\Psi_b^{\langle-1\rangle}(w)}{\Psi_{a,b}(\Psi_a^{\langle-1\rangle}(z),\Psi_b^{\langle-1\rangle}(w))+w}\\
                                                      0 & \Psi_b^{\langle-1\rangle}(w)
                                                    \end{matrix}\right)\\
=&\left(\begin{matrix}S_a(z)&\frac{\zeta(z+1)}{z}\frac{\Psi_b^{\langle-1\rangle}(w)}{\Psi_{a,b}(\Psi_a^{\langle-1\rangle}(z),\Psi_b^{\langle-1\rangle}(w))+w}
-\frac{\zeta\Psi_b^{\langle-1\rangle}(w)}{zw}\\
0&S_b(w)\end{matrix}\right).\\
\end{align*}
Let $X_i=\left (\begin{matrix}a_i&0\\0&b_i\end{matrix}\right)$, for $i=1, 2$. The $(1,2)$ entry of $S_{X_1X_2}$ has the following form
\begin{align*}
&\frac{\zeta(z+1)}{z}\frac{\Psi_{b_1b_2}^{\langle-1\rangle}(w)}{\Psi_{a_1a_2,b_1b_2}(\Psi_{a_1a_2}^{\langle-1\rangle}(z),\Psi_{b_1b_2}^{\langle-1\rangle}(w))+w}
-\frac{\zeta\Psi_{b_1b_2}^{\langle-1\rangle}(w)}{zw}\\
=&\frac{\zeta(z+1)\Psi_{b_1b_2}^{\langle-1\rangle}(w)}{z}\left( \frac{1}{\Psi_{a_1a_2,b_1b_2}(\Psi_{a_1a_2}^{\langle-1\rangle}(z),\Psi_{b_1b_2}^{\langle-1\rangle}(w))+w}-\frac{1}{w(z+1)}\right)\\
=&\frac{\zeta(z+1)\Psi_{b_1b_2}^{\langle-1\rangle}(w)}{z}\left( \frac{wz+w-\Psi_{a_1a_2,b_1b_2}(\Psi_{a_1a_2}^{\langle-1\rangle}(z),\Psi_{b_1b_2}^{\langle-1\rangle}(w))-w}{(\Psi_{a_1a_2,b_1b_2}(\Psi_{a_1a_2}^{\langle-1\rangle}(z),\Psi_{b_1b_2}^{\langle-1\rangle}(w))+w)w(z+1)}\right)\\
=&\frac{\zeta\Psi_{b_1b_2}^{\langle-1\rangle}(w)}{w}\left( \frac{\frac{wz}{\Psi_{a_1a_2,b_1b_2}(\Psi_{a_1a_2}^{\langle-1\rangle}(z),\Psi_{b_1b_2}^{\langle-1\rangle}(w))}-1}
{z+\frac{zw}{\Psi_{a_1a_2,b_1b_2}(\Psi_{a_1a_2}^{\langle-1\rangle}(z),\Psi_{b_1b_2}^{\langle-1\rangle}(w))}}\right).
\end{align*}
Since $\lim_{w\rightarrow 0}\frac{\Psi_{b_1b_2}^{\langle-1\rangle}(w)}{w}=\frac{1}{\varphi(b_1b_2)}$, and
\begin{align*}
&\lim_{z,w\rightarrow 0}\frac{zw}{\Psi_{a_1a_2,b_1b_2}(\Psi_{a_1a_2}^{\langle-1\rangle}(z),\Psi_{b_1b_2}^{\langle-1\rangle}(w))}\\
=&\lim_{z,w\rightarrow 0}\frac{\Psi_{a_1a_2}^{\langle-1\rangle}(z)\Psi_{b_1b_2}^{\langle-1\rangle}(w)}{\Psi_{a_1a_2,b_1b_2}(\Psi_{a_1a_2}^{\langle-1\rangle}(z),\Psi_{b_1b_2}^{\langle-1\rangle}(w))}
\lim_{z,w\rightarrow 0}\frac{zw}{\Psi_{a_1a_2}^{\langle-1\rangle}(z)\Psi_{b_1b_2}^{\langle-1\rangle}(w)}\\
=&\frac{1}{\varphi(a_1a_2b_1b_2)}\varphi(a_1a_2)\varphi(b_1b_2),
\end{align*}
we get that, when $z, w\rightarrow 0$, the $(1,2)$ entry of $S_{X_1X_2}$ has a limit $\zeta \frac{\varphi(a_1a_2)\varphi(b_1b_2)-\varphi(a_1a_2b_1b_2)}{\varphi(a_1a_2)\varphi(b_1b_2)^2}$.

Let $x=x^0+\varphi(x)$, for $x\in \A$, where $x^0=x-\varphi(x)$. Note that $(a_1, b_1)$ and $(a_2, b_2)$ are bi-free, by Lemma 2.1 of \cite{DV2}, we have
\begin{align*}
\varphi(a_1a_2b_1b_2)=&\varphi((a_1^0+\varphi(a_1))(a_2^0+\varphi(a_2))(b_1^0+\varphi(b_1))(b_2^0+\varphi(b_2)))\\
=&\varphi(a_1^0a_2^0b_1^0b_2^0)+\varphi(a_1)\varphi(a_2)\varphi(b_1)\varphi(b_2)+\varphi(a_1)\varphi(b_1)\varphi(a_2^0b_2^0)+\varphi(a_2)\varphi(b_2)\varphi(a_1^0b_1^0)\\
=&\varphi(a_1^0b_1^0)\varphi(a_2^0b_2^0)+\varphi(a_1)\varphi(a_2)\varphi(b_1)\varphi(b_2)+\varphi(a_1)\varphi(b_1)\varphi(a_2^0b_2^0)+\varphi(a_2)\varphi(b_2)\varphi(a_1^0b_1^0)\\
=&(\varphi(a_1^0b_1^0)+\varphi(a_1b_1))\varphi(a_2^0b_2^0)+\varphi(a_2)\varphi(b_2)(\varphi(a_1^0b_1^0)+\varphi(a_1b_1))=\varphi(a_1b_1)\varphi(a_2b_2).
\end{align*}
On the other hand, by the proof of Theorem 4.1 in \cite{DV3}, we have
$$\lim_{z\rightarrow 0}S_{a}(z)=\lim_{z\rightarrow 0}\frac{1+z}{z}\Psi_a^{\langle -1\rangle}(z)=\varphi(a)^{-1}. $$

The conclusive equation $(2.2)$ now follows from the above calculations, and the fact that $a_1$ and $a_2$, and $b_1$ and $b_2$ are two free pairs of random variables whenever $\{a_1, b_1\}$ and $\{a_2, b_2\}$ are bi-free.
\end{proof}

\begin{Lemma}If $(a_1,b_1)$ and $(a_2, b_2)$ are bi-free in $(\A, \varphi)$ and $\varphi(a_j)\varphi(b_j)\ne 0$, for $j=1,2$,   then we have \begin{align*}&\lim_{z\rightarrow 0, w\rightarrow 0}S_{X_2}(\Gamma)S_{X_1}(S_{X_2}(\Gamma)^{-1}\Gamma S_{X_2}(\Gamma))\\
=&\left(\begin{matrix}\frac{1}{\varphi(a_1)\varphi(a_2)}&\zeta \left(\frac{\varphi(a_1)\varphi(b_1)-\varphi(a_1b_1)}{\varphi(a_1)\varphi(b_1)^2\varphi(b_2)}+\frac{\varphi(a_2)\varphi(b_2)-\varphi(a_2b_2)}{\varphi(a_2)\varphi(b_2)^2\varphi(b_1)}\right)\\
0&\frac{1}{\varphi(b_1)\varphi(b_2)}\end{matrix}\right).
\end{align*}
\end{Lemma}
\begin{proof}
 Let $\zeta\rho_1(z,w)$ and $\zeta\rho_2(z,w)$ denote the $(1.2)$ entries of  $S_{X_1}(\Gamma)$ and $S_{X_2}(\Gamma)$,  respectively.  Since $(a_1, b_1)$ and $(a_2, b_2)$ are bi-free, we have
\begin{align*}
&S_{X_2}(\Gamma)S_{X_1}(S_{X_2}(\Gamma)^{-1}\Gamma S_{X_2}(\Gamma))=\left(\begin{matrix}S_{a_2}(z)&\zeta \rho_2(z,w)\\0&S_{b_2}(w)\end{matrix}\right)\\
\times &S_{X_1}\left(\left(\begin{matrix}S_{a_2}(z)^{-1}&-\zeta\rho_2(z,w)S_{a_2}(z)^{-1}S_{b_2}(w)^{-1}\\ 0&S_{b_2}(w)^{-1}\end{matrix}\right)\left(\begin{matrix}z&\zeta \\0&w\end{matrix}\right)\left(\begin{matrix}S_{a_2}(z)&\zeta \rho_2(z,w)\\0&S_{b_2}(w)\end{matrix}\right)\right)\\
=&\left(\begin{matrix}S_{a_2}(z)&\zeta \rho_2(z,w)\\0&S_{b_2}(w)\end{matrix}\right)\\
\times &S_{X_1}\left(\left(\begin{matrix}zS_{a_2}(z)^{-1}&\zeta S_{a_2}(z)^{-1}-\zeta w\rho_2(z,w)S_{a_2}(z)^{-1}S_{b_2}(w)^{-1}\\ 0&wS_{b_2}(w)^{-1}\end{matrix}\right)\left(\begin{matrix}S_{a_2}(z)&\zeta \rho_2(z,w)\\0&S_{b_2}(w)\end{matrix}\right)\right)\\
=&\left(\begin{matrix}S_{a_2}(z)&\zeta \rho_2(z,w)\\0&S_{b_2}(w)\end{matrix}\right)
S_{X_1}\left(\left(\begin{matrix}z&\zeta zS_{a_2}(z)^{-1}\rho_2(z,w)+\zeta S_{a_2}(z)^{-1}S_{b_2}(w)-\zeta w\rho_2(z,w)S_{a_2}(z)^{-1}\\ 0&w\end{matrix}\right)\right)\\
=&\left(\begin{matrix}S_{a_2}(z)&\zeta \rho_2(z,w)\\0&S_{b_2}(w)\end{matrix}\right)\left(\begin{matrix}S_{a_1}(z)&\zeta((z-w)S_{a_2}^{-1}(z)\rho_2(z,w)+S_{a_2}(z)^{-1}S_{b_2}(w))\rho_1(z,w)\\0&S_{b_1}(w)\end{matrix}\right)\\
=&\left(\begin{matrix}S_{a_2}(z)S_{a_1}(z)&\zeta\rho_1(z,w)((z-w)\rho_2(z,w)+S_{b_2}(w))+\zeta\rho_2(z,w)S_{b_1}(w)\\0&S_{b_2}(w)S_{b_1}(w)\end{matrix}\right).
\end{align*}
Let $\gamma(z,w)=\rho_1(z,w)((z-w)\rho_2(z,w)+S_{b_2}(w))+\rho_2(z,w)S_{b_1}(w)$. We then get
$$\lim_{z\rightarrow 0, w\rightarrow 0}\gamma(z,w)=\frac{\varphi(a_1)\varphi(b_1)-\varphi(a_1b_1)}{\varphi(a_1)\varphi(b_1)^2\varphi(b_2)}+\frac{\varphi(a_2)\varphi(b_2)-\varphi(a_2b_2)}{\varphi(a_2)\varphi(b_2)^2\varphi(b_1)}.$$
\end{proof}

\begin{Theorem} Suppose that $(a_1, b_1)$ and $(a_2, b_2)$ are bi-free in $(\A, \varphi)$, and $\varphi(a_j)\varphi(b_j)\ne 0$, for $j=1,2$. Then $(2.1)$ holds true if and only if either $S_{a_1, b_1}(z,w)\equiv 1$ or $S_{a_2, b_2}(z,w)\equiv1$ for $z$ and $w$ near zero.
\end{Theorem}
\begin{proof}
We use the symbols defined in Lemmas 2.1 and 2.2 $$\rho_{a,b}(z,w)=\frac{S_b(w)(1-S_{a.b}(z,w))}{zS_{a,b}(z,w)+w+1},\ S_j=S_{a_j,b_j}(z,w), j=1,2;$$
$$\rho_j=\rho_{a_j,b_j}(z,w), j=1,2, \ \gamma(z,w)=(z-w)\rho_1\rho_2+\rho_1S_{b_2}(w)+\rho_2S_{b_1}(w).$$ By the proofs of Lemmas 2.1 and 2.2, $(2.1)$ holds true if and only if $\rho_{a_1a_2, b_1b_2}(z,w)=\gamma(z,w)$ when $z$ and $w$ are near zero.

Without loss of generality, we suppose that $S_{a_1,b_1}(z,w)\equiv1$ when $z$ and $w$ are near zero.  Then

$$\rho_{a_1a_2, b_1b_2}(z,w)=S_{b_1}(w)\rho_2(z,w)=\gamma(z,w),$$
 when $z$ and $w$ are near zero.

Conversely, suppose that $\rho_{a_1a_2, b_1b_2}(z,w)=\gamma(z,w)$ when $z$ and $w$ are near zero. We then have
\begin{align*}
\frac{1-S_1S_2}{zS_1S_2+w+1}=&\frac{(z-w)(1-S_1)(1-S_2)}{(zS_1+w+1)(zS_2+w+1)}+\frac{1-S_1}{zS_1+w+1}+\frac{1-S_2}{zS_2+w+1}\\
=&\frac{(z-w)(1-S_1)(1-S_2)+(1-S_1)(zS_2+w+1)+(1-S_2)(zS_1+w+1)}{(zS_1+w+1)(zS_2+w+1)}\\
=&\frac{-(z+w)S_1S_2-S_1-S_2+z+w+2}{(zS_1+w+1)(zS_2+w+1)}.
\end{align*}
It implies that
\begin{align*}
&(1-S_1S_2)(zS_1+w+1)(zS_2+w+1)\\
=&(1-S_1S_2)(z^2S_1S_2+z(w+1)(S_1+S_2)+(w+1)^2)\\
=&-z^2(S_1S_2)^2-z(w+1)S_1^2S_2-z(w+1)S_1S_2^2+(z^2-w^2-2w-1)S_1S_2\\
+&z(w+1)S_1+z(w+1)S_2+(w+1)^2\\
=&(zS_1S_2+w+1)(-(z+w)S_1S_2-S_1-S_2+z+w+2)\\
=&-z(w+z)(S_1S_2)^2-zS_1^2S_2-zS_1S_2^2+(z^2+z-w^2-w)S_1S_2\\
+&(w+1)(z+w+2)-(w+1)S_1-(w+1)S_2.
\end{align*}
Therefore,
\begin{align*}
0=&-zw(S_1S_2)^2+zwS_1^2S_2+zwS_1S_2^2+(z+w+1)S_1S_2\\
-&(z+1)(w+1)(S_1+S_2)+(1+z)(1+w)\\
=&S_1S_2(z+w+1+zw(S_1+S_2)-zwS_1S_2)+(1+z)(1+w)(1-S_1-S_2)\\
=&S_1S_2((z+1)(w+1)-zw(1-S_1)(1-S_2))+(1+z)(1+w)(1-S_1-S_2)\\
=&(z+1)(w+1)\left(S_1S_2\left(1-\frac{zw}{(1+z)(1+w)}(1-S_1)(1-S_2)\right)+1-S_1-S_2\right).
\end{align*}
When $z$ and $w$ are very close to zero, we get
$$S_1S_2\left(1-\frac{zw}{(1+z)(1+w)}(1-S_1)(1-S_2)\right)+1-S_1-S_2=0.$$
 We can rewrite the above equation as
 $$(1-S_1)(1-S_2)\left(1-S_1S_2\frac{zw}{(1+z)(1+w)}\right)=0.$$
 On the other hand, by Voiculescu's multiplicative formula for bi-free partial $S$-transforms (1.1), we have
 \begin{align*}
 \lim_{z\rightarrow 0, w\rightarrow 0}S_1S_2\frac{zw}{(1+z)(1+w)}=&\lim_{z\rightarrow 0, w\rightarrow 0}S_{a_1a_2, b_1b_2}(z,w)\frac{zw}{(1+z)(1+w)}\\
 =&\lim_{z\rightarrow 0, w\rightarrow 0}\frac{\Psi_{a_1a_2,b_1b_2}(\Psi_{a_1a_2}^{\langle-1\rangle}(z),\Psi_{b_1b_2}^{\langle-1\rangle}(w) )}{H_{a_1a_2,b_1b_2}(\Psi_{a_1a_2}^{\langle-1\rangle}(z),\Psi_{b_1b_2}^{\langle-1\rangle}(w))}\\
 =&0.
 \end{align*}
 It follows that $$(1-S_1)(1-S_2)=0,$$ for $(z,w)\in D_r=\{(z,w):|z|<r, |w|<r\}$ for some $r>0$.  By Proposition 4.1 in \cite{DV3}, $S_1$ and $S_2$ are holomorphic functions of $(z,w)$ in a neighborhood of $(0,0)$. If $S_1(z,w)$ is not the constant function 1 in $D_r$, by Lemma 24 in \cite{PG}, $\{(z,w)\in D_r: 1-S_1(z,w)=0\}$ is a nowhere dense subset of $D_r$. It implies that there exists an open subset $U\subseteq D_r$ such that $1-S_(z,w)\ne 0$, for all $(z,w)\in U$. Therefore, $1-S_2(z,w)=0$, for all $(z,w)\in U$. By the uniqueness theorem of holomorphic functions of several complex variables (Theorem 1 in \cite{PG}), $S_2(z,w)=1$, for all $(z,w)\in D_r$.
\end{proof}

\begin{Remark}A pair $(a,b)$ of random variables in a non-commutative probability space  $(\A, \varphi)$ is said to {\sl have factoring two-band moments}, if $\varphi(a^mb^n)=\varphi(a^m)\varphi(b^n)$, for all $m,n=1, 2, \cdots$. By Remark 4.4 in \cite{PS1}, Proposition 4.2 in \cite{DV3}, or Remark 2.7 in \cite{PS2}, $(a,b)$ has factoring two-band moments, if and only if $S_{a,b}(z,w)\equiv 1$ when $z$ and $w$ are very close to the origin $0$ in $\mathbb{C}$. Therefore, we get the following corollary.
\end{Remark}

\begin{Corollary} If $(a_1, b_1)$ and $(a_2, b_2)$ are bi-free in $(\A, \varphi)$, and $$\varphi(a_jb_j)\ne \varphi(a_j)\varphi(b_j),\  \varphi(a_j)\ne 0\ne \varphi(b_j),$$ for $j=1,2$, then $$S_{X_1X_2}(\Gamma)\ne S_{X_2}S_{X_1}(S_{X_2}(\Gamma)^{-1}\Gamma S_{X_2}(\Gamma)),$$ in any $\varepsilon$-neighborhood $\{\Gamma: \|\Gamma\|<\varepsilon\}$ of matrix $0$.
\end{Corollary}
The above result shows that there are many  pairs $(a_1,b_1)$ and $(a_2,b_2)$ for which $S_{X_1}$ and $S_{X_2}$ do not satisfy $(2.1)$. We shall give such an example derived from $W^*$-free products of finite von Neumann algebras.
\begin{Example}[6.2 in \cite{DV1}]Let $(\A, \tau)$ be the $W^*$-free product of two type $II_1$ factors $\A_1$ and $\A_2$ with the faithful normal tracial states $\tau_1$ and $\tau_2$ on $\A_1$ and $\A_2$, respectively. Let $$\varphi: B(L^2(\A, \tau))\rightarrow \mathbb{C},\  \varphi(T)=\langle T1,1\rangle, \forall T\in B(L^2(\A, \tau)).$$ Let $L:\A\rightarrow B(L^2(\A, \tau))$, $R:\A^{op}\rightarrow B(L^2(\A,\tau))$ be the regular left and right presentations $$L(m)h=mh, \ R(m)h=hm, m\in \A, h\in L^2(\A, \tau),$$ and  $$L_j=L|_{\A_j}:\A_j\rightarrow B(L^2(\A,\tau)), \ R_j=R|_{\A_j^{op}}:\A_j^{op}\rightarrow B(L^2(\A,\tau)).$$ By 6.2 in \cite{DV1}, $(L_1(\A_1), R_1(\A_1^{op}))$ and $((L_2(\A_2), R_2(\A_2^{op})))$ are bi-free in $(B(L^2(\A,\tau)), \varphi)$.

Choose $x_j, y_j \in \A_j$ such that $\tau_j(x_jy_j)\ne \tau_j(x_j)\tau_j(y_j)$, for $j=1,2$. Let $a_j=L_j(x_j)$ and $b_j=R(y_j)$, for $j=1,2$. Then
\begin{align*}
\varphi(a_jb_j)=&\langle x_jy_j1,1\rangle=\tau(x_jy_j)=\tau_j(x_jy_j)\ne \tau_j(x_j)\tau_j(y_j)\\
=&\langle L_j(x_j)1,1\rangle\langle L_j(y_j)1,1\rangle=\varphi(a_j)\varphi(b_j),
 \end{align*}
 for $j=1,2$. It follows from Corollary 2.5 that $S_{X_1X_2}(\Gamma)\ne S_{X_2}(\Gamma)S_{X_1}(S_{X_2}(\Gamma)^{-1}\Gamma S_{X_2}(\Gamma))$, for $z, w$ in any neighborhood of zero.
\end{Example}

If  both $(a_1,b_1)$ and $(a_2,b_2)$ have factoring two-band moments, we can get  a subordination result for the $\Psi$-transforms of $X_1$, $X_2$, and  $X_1X_2$.
\begin{Theorem}If both $(a_1,b_1)$ and $(a_2,b_2)$ have factoring two-band moments, and $(a_1, b_1)$ and $(a_2, b_2)$ are bi-free in $(\A, \varphi)$,  $a_j, b_j$ are unitaries (or non-zero positive elements) in a $C^*$-probability space $(\A, \varphi)$, for $j=1,2$, then there are analytic functions $\omega_{a_i}, \omega_{b_i}:\mathbb{D}\rightarrow \mathbb{D}$ (or $\omega_{a_i}, \omega_{b_i}:\mathbb{C}\setminus \mathbb{R}^+\rightarrow \mathbb{C}\setminus \mathbb{R}^+$ ) such that $$\Psi_{X_1X_2}(\Gamma)=\Psi_{X_j}(\omega_j(\Gamma)),$$ whenever $|z|$ and $|w|$ are small enough, where $$\omega_j(\Gamma)=\left(\begin{matrix}\omega_{a_j}(z)&\frac{\zeta\omega_{b_j}(w)}{w}\\0&\omega_{b_j}(w)\end{matrix}\right),$$ $j=1,2$.
\end{Theorem}
\begin{proof}
By Remark 2.4, $S_{a_j,b_j}(z,w)=1, \forall z,w\in D_r, j=1,2$,   where $D_r$ is the open disk in $\mathbb{C}$ centered at $0$ with some radius $r>0$.
Therefore, $$S_{a_1a_2, b_2b_2}(z,w)=S_{a_1,b_1}(z,w)S_{a_2,b_2}(z,w)=1,$$ for $z,w$ near zero, since $(a_1, b_1)$ and $(a_2, b_2)$ are bi-free. By Remark 2.4 again,  $$\varphi((a_1a_2)^m(b_1b_2)^n)=\varphi((a_1a_2)^m)\varphi((b_1b_2)^n), m,n\in \mathbb{N}.$$
It implies that $$\Psi_{a_1a_2,b_1b_2}(z,w)=\Psi_{a_1a_2}(z)\Psi_{b_1b_2}(w), \ \Psi_{a_j,b_j}(z,w)=\Psi_{a_j}(z)\Psi_{b_j}(w), j=1, 2, \eqno (2.3)$$ whenever $|z|$ and $|w|$ are mall enough. By (2.3) and the proof of Lemma 2.1, we get $$\Psi_{X_1X_2}(\Gamma)=\left(\begin{matrix}\Psi_{a_1a_2}(z)&\frac{\zeta}{w}\Psi_{b_1b_2}(w)(\Psi_{a_1a_2}(z)+1)\\0&\Psi_{b_1b_2}(w)\end{matrix}\right),\ \Psi_{X_j}(\Gamma)
=\left(\begin{matrix}\Psi_{a_j}(z)&\frac{\zeta}{w}\Psi_{b_j}(w)(\Psi_{a_j}(z)+1)\\0&\Psi_{b_j}(w)\end{matrix}\right),$$ for $j=1, 2$, whenever $|z|$ and $|w|$ are mall enough.

Since $a_j, b_j$ are unitaries (or non-zero positive elements) in a $C^*$-probability space $(\A, \varphi)$, for $j=1,2$, by the subordination theorems for multiplicative free convolution (Theorems 3.2 and 3.3 in \cite{BB}),    there are analytic functions $\omega_{a_i}, \omega_{b_i}:\mathbb{D}\rightarrow \mathbb{D}$ (or $\omega_{a_i}, \omega_{b_i}:\mathbb{C}\setminus \mathbb{R}^+\rightarrow \mathbb{C}\setminus \mathbb{R}^+$ ) such that $$\Psi_{a_1a_2}(z)=\Psi_{a_i}(\omega_{a_i}(z)), \ \Psi_{b_1b_2}(z)=\Psi_{b_i}(\omega_{b_i}(z)),$$ for $z\in \mathbb{D}$ (or $z\in \mathbb{C}\setminus \mathbb{R}^+$),  and $i=1,2$. Therefore,
$$\Psi_{X_1X_2}(\Gamma)=\left(\begin{matrix}\Psi_{a_j}(\omega_{a_j}(z))&\frac{\zeta}{w}\Psi_{b_j}(\omega_{b_j}(w))(\Psi_{a_j}(\omega_{a_j}(z))+1)\\0&\Psi_{b_j}(\omega_{b_j}(w))\end{matrix}\right)=\Psi_{X_j}(\omega_j(\Gamma)),$$ for $j=1,2$.
\end{proof}
\begin{Remark} It is obvious that if $a$ and $b$ are freely independent, or tensorial independent, then $\varphi(a^mb^n)=\varphi(a^m)\varphi(b^n)$, for all $m,n\in \mathbb{N}$. But, conversely, the property of factoring two-band moments does not imply freely or tensorial independence of $a$ and $b$.
\end{Remark}


\begin{thebibliography}{99}

\bibitem[BB]{BB}S. T. Belinschi and H. Bercovici. {\sl A New Approach to Subordination Results in Free Probability}. J. D'Analyse Math., Vol. 101(2007), 357-365.
\bibitem[BBGS]{BBGS}S. T. Berlinschi, H. Bercovici, Y. Gu, and P. Skoufranis. {\sl Analytic Subordination for Bi-free convolution}. arXiv:1702.01673v1, [math.OA], 6 Feb., 2017.
\bibitem[BSTV]{BSTV} S. T. Belinschi, R. Speicher, J. Treilhard, and C. Vargas. {\sl Operator-valued Free Multiplicative convolution: Analytic Subordination theory and Applications to Random Matrix Theory}. IMRN Vol. 2015, Issue 14, 5933-5958, Jan. 2015.
\bibitem[PB]{PB}P. Biane. {\sl Processes with free increments}. Math. Z., 227(1998), 143-174.
\bibitem[CNS1]{CNS1}I. Charlesworth, B. Nelson, and P. Skoufranis. {\sl On Two-faced Families of non-commutative random variables}. Canad. J. Math., 26(2015), no.6, 1290-1325.
\bibitem[KD]{KD}K. Dykema. {\sl On the $S$-transform over a Banach algebra}. J. Funct. Anal., 231(2006), 90-110.
\bibitem[PG]{PG}P. M. Gauthier. {\sl Lectures on Several Complex Variables}. Springer International Publishing Swizerland 2014.
\bibitem[GHM]{GHM}Y. Gu, H. Huang, and J. Mingo. {\sl An analogue of the Levy-Hincin formula for bi-free infintely divisible distributions}.  Indiana Univer. Math. J. volume 65, issue 5, 2016, pp.  1795-1831.
\bibitem[HW]{HW} H.-W. Huang and J.-C. Wang. {\sl Harmonic Analysis for the Bi-free Partial $S$-transforms}. arXiv:1705.06569v1, [math.OA], 28, Apr., 2017.
\bibitem[NS]{NS}A. Nica and R. Speicher. {\sl Lectures on combinatorics for free probbaility}. London Math. Soc. Lecture Notes Series Vol. 335, Cambridge Univ. Press, 2006.
\bibitem[PS1]{PS1}P. Skoufranis. {\sl A Combinatorial Approach to Voiculescu’s Bi-Free Partial Transforms}. Pacific J. Math. 283 (2016), no. 2, 419-447.
\bibitem[PS2]{PS2}P. Skoufranis. {\sl A Combinatorial Approach to the Opposite Bi-Free Partial $S$-Transforms}.  arXiv:1705.02857v1 [math.OA]8 May, 2017.
\bibitem[DV1]{DV1}D. Voiculescu. {\sl Free probability of pairs of two faces I}. Comm. Math. Phys. 332(2014), 955-980.
\bibitem[DV2]{DV2}D. Voiculescu. {\sl Free probability of pairs of two faces II}. Ann. Inst. Henri Poincare Probab. Stat. 52(2016), No. 1, 1-15.
\bibitem[DV3]{DV3}D. Voiculescu. {\sl Free Probability for Pairs of Faces III: 2-variable Bi-free Partial $S$- and $T$- transforms}. J. Funct. Anal., 270(2016), No. 10, 3623-3638.
\bibitem[DV4]{DV4}D. Voiculescu. {\sl Multipliation of Certain Non-Commutative Random Variables}. J. Operator Th., 18(1987), 223-235.
\end{thebibliography}
\end{document}